\documentclass[11pt]{article}
\usepackage{latexsym,amsmath,color,amsthm,amssymb,epsfig,graphicx,mathrsfs}
\usepackage{graphicx}
\usepackage{amssymb}
\usepackage[left=1in,top=1in,right=1in,bottom=1in]{geometry}
\usepackage[linktocpage=true]{hyperref}
\usepackage{setspace}
\usepackage{amssymb, amsmath, amsthm, graphicx,mathrsfs}
\usepackage{caption}

\usepackage{comment}
\def\qed{\hfill\ifhmode\unskip\nobreak\fi\quad\ifmmode\Box\else\hfill$\Box$\fi}
\def\ite#1{\hfill\break${}$\hbox to 50pt {\quad(#1)\hfill}}

\def\cG{{\mathcal G}}
\def\cH{{\mathcal H}}

\newtheorem{thm}{Theorem}[section]
\newtheorem{cor}[thm]{Corollary}
\newtheorem{const}[thm]{Example}
\newtheorem{definition}[thm]{Definition}

\newtheorem{lem}[thm]{Lemma}
\newtheorem{conj}[thm]{Conjecture}

\newtheorem{claim}[thm]{Claim}

\begin{document}

\pagestyle{myheadings} 
\markright{{\small{\sc A.~Kostochka,  R. Luo, and D. Zirlin:   Super-pancyclic hypergraphs}}}

\title{\vspace{-0.5in} Super-pancyclic hypergraphs and bipartite graphs}

\author{
{{Alexandr Kostochka}}\thanks{
\footnotesize {University of Illinois at Urbana--Champaign, Urbana, IL 61801
 and Sobolev Institute of Mathematics, Novosibirsk 630090, Russia. E-mail: \texttt {kostochk@math.uiuc.edu}.
 Research 
is supported in part by NSF grant  DMS-1600592
and grants  18-01-00353A and 19-01-00682
  of the Russian Foundation for Basic Research.
}}
\and{{Ruth Luo}}\thanks{University of Illinois at Urbana--Champaign, Urbana, IL 61801, USA. E-mail: {\tt ruthluo2@illinois.edu}.
Research 
is supported in part by NSF grant  DMS-1600592.
}
\and{{Dara Zirlin}}\thanks{University of Illinois at Urbana--Champaign, Urbana, IL 61801, USA. E-mail: {\tt zirlin2@illinois.edu}.}}

\date{\today}

\maketitle

\vspace{-0.3in}

\begin{abstract}
 We find Dirac-type sufficient conditions for a hypergraph $\cH$ with few edges  to be hamiltonian. 
 We also show that these conditions provide that $\cH$ 
  is {\em super-pancyclic}, i.e.,   for each  $A \subseteq V(\cH)$ with $|A| \geq 3$, $\cH$ contains a Berge cycle
with vertex set $A$.

We  mostly use the language of bipartite graphs, because every bipartite graph is the incidence graph of a multihypergraph. In particular, we extend some results of Jackson on the existence of long cycles in bipartite graphs where the vertices in one part have high minimum degree. Furthermore, we prove a conjecture of Jackson from 1981 on long cycles in 2-connected bipartite graphs.

\medskip\noindent
{\bf{Mathematics Subject Classification:}} 05D05, 05C65, 05C38, 05C35.\\
{\bf{Keywords:}} Berge cycles, extremal hypergraph theory, bipartite graphs.
\end{abstract}

\section{Introduction}

\subsection{Cycles in bipartite graphs}
For integers $n, m,$ and $\delta$ with $\delta \leq m$, we denote by $\cG(n,m,\delta)$ the set of all bipartite graphs with partition $(X, Y)$ such that $|X| = n\geq 2, |Y|=m$ and for every $x \in X$, $d(x) \geq \delta$. Jackson~\cite{jackson,jackson2} proved the following theorem (among several other results) on the existence of long cycles in graphs $\cG(n,m,\delta)$. 

\begin{thm}[Jackson~\cite{jackson}]\label{jackson} If a graph $G\in \cG(n,m,\delta)$ satisfies $n \leq \delta$ and $m \leq 2\delta-2$, then it contains a cycle of length $2n$, i.e., a cycle that covers $X$.\end{thm}

The following two constructions with $m = 2\delta-1$ show that Theorem~\ref{jackson} is best possible.

\begin{const}\label{con1}
For $\delta=n$, let $G_1(n) \in \cG(\delta,2\delta-1,\delta)$ be obtained from a copy of $K_{\delta,\delta-1}$ where every vertex in $X$ has an additional neighbor of degree 1. Then the longest cycle of $G_1(n)$ has length $2(n-1)$.\end{const}

\begin{const}\label{con2}
Fix positive integers $a \geq b$ such that $a + b = n$. Let $G_2(a,b) \in \cG(n,2\delta-1, \delta)$ be the bipartite graph obtained from a copy $H_1$ of $K_{a, \delta}$ and a copy $H_2$ of $K_{b,\delta}$ by gluing together a vertex of $H_1$ in a  part of size $\delta$
and a vertex of $H_2$ in a  part of size $\delta$. Then the longest cycle of $G_2(a,b)$ has length $2a \leq 2(n-1)$. 
\end{const}

Both examples above contain cut vertices. Jackson  conjectured that if  $G \in \cG(n,m,\delta)$ is 2-connected, then we may relax the upper bound on $m$.

\begin{conj}[Jackson~\cite{jackson}]\label{jacksonconj} Let $m,n,\delta$ be integers. If some graph $G \in \cG(n,m,\delta)$ is 2-connected and satisfies 
\begin{enumerate}
\item[(i)]$m \leq 3\delta-5$ if $n \leq \delta$, or
\item[(ii)] $m\leq \lfloor \frac{2(n-\alpha)}{\delta-1-\alpha} \rfloor (\delta-2) + 1$ if $n \geq \delta$,
\end{enumerate}
where $\alpha = 1$ if $\delta$ is even and $\alpha = 0$ if $\delta$ is odd, then $G$ contains a cycle of length $2\min(n,\delta)$.\end{conj}

If true, both (i) and (ii) would be best possible. 

\begin{const}\label{con5}
For (i), fix positive integers $n_1 \geq n_2 \geq n_3$ such that $n_1 + n_2 + n_3 = n$.
 Let $G_3(n_1, n_2, n_3) \in \cG(n,3\delta-4, \delta)$ be the bipartite graph obtained from $K_{\delta-2, n_1} \cup K_{\delta-2, n_2} \cup K_{\delta-2, n_3}$  by adding two vertices $a$ and $b$ that are both adjacent to every vertex in the parts of size $n_1, n_2$, and $n_3$. Then a longest cycle in $G_3(n_1, n_2, n_3)$ has length $2(n_1+n_2) \leq 2(n-1)$.  
 \end{const}
 
 An extremal construction for (ii) is given in~\cite{jackson}.

In this paper we prove part (i) of Conjecture~\ref{jacksonconj}.

\begin{thm}\label{mainj} Fix integers $n,m,\delta$ such that $n \leq \delta \leq m \leq 3\delta-5$. If $G \in \cG(n,m,\delta)$ is 2-connected, then $G$ contains a cycle of length $2n$, i.e., a cycle that covers $X$.
\end{thm}

 We also prove the following extension of Theorem~\ref{jackson}.

\begin{thm}\label{jackson2} Let $\delta\geq n$ and $m\leq 2\delta-1$. If $G \in \cG(n,m, \delta)$ does not contain a cycle of length $2n$, then 
 either $G = G_1(n)$
in Example~\ref{con1} or $G = G_2(a,b)$ for some  $a$ and $b$ with $a + b = n$ in Example~\ref{con2}. 
\end{thm}

The setup of  bipartite graphs with large degrees in one part is  useful for finding long cycles in hypergraphs with high minimum degree.

\subsection{Berge cycles in hypergraphs}

A {\em hypergraph} $\cH$ is a set of vertices $V(\cH)$ and a set of edges $E(\cH)$ such that each edge is a subset of $V(\cH)$. Often we take $V(\cH) = [n]$ where $[n]:=\{1,\ldots, n\}$ is the set of the first $n$ integers.

We consider hypergraphs without restrictions on edge cardinality (i.e., our hypergraphs may have edges of any size). The degree of a vertex $v$, denoted $d(v)$, is the number of edges that contain $v$. The minimum degree of a hypergraph $\cH$ is denoted $\delta(\cH) := \min_{v\in V(\cH)} d(v)$. The co-degree of a set of vertices $\{v_1, \ldots, v_\ell\}$ 
 is the number of edges that contain the set $\{v_1, \ldots, v_\ell\}$.

\begin{definition} 
A {\bf Berge cycle} of length $\ell$ in a hypergraph is a set of $\ell$ distinct vertices $\{v_1, \ldots, v_\ell\}$ and $\ell$ distinct edges $\{e_1, \ldots, e_\ell\}$ such that for every $i\in [\ell]$, $v_i, v_{i+1} \in e_i$ with indices taken modulo $\ell$. The vertices $\{v_1, \ldots, v_\ell\}$ are called the {\bf base vertices} of the Berge cycle.

Furthermore, a {\bf Berge hamiltonian cycle} in  a hypergraph $\cH$ is a Berge cycle whose set of base vertices is $V(\cH)$.  
\end{definition}


\begin{definition}Let $\cH = (V(\cH), E(\cH))$ be a hypergraph. The {\bf incidence   graph  } of $\cH$ is the bipartite graph $I(\cH)$ with parts $(X, Y)$ where $X = V(\cH)$, $Y= E(\cH)$ such that for $e \in Y, v \in X,$ $ev\in E(I(\cH))$ if and only if the vertex $v$ is contained in the edge $e$ in $\cH$. 
\end{definition}

If $\cH$ has $n$ vertices, $m$ edges and minimum degree at least $\delta$, then
 $I(\cH) \in \cG(n,m,\delta)$. Furthermore, if $\{v_1, \ldots, v_\ell\}$ and $\{e_1, \ldots, e_\ell\}$ form a Berge cycle of length $\ell$ in $\cH$, then $v_1 e_1 \ldots v_\ell e_\ell v_1$ is a cycle of length $2\ell$ in $I(\cH)$, and vice versa. 

Using incidence   graphs, we also define 2-connectedness in hypergraphs. 

\begin{definition} A hypergraph is {\bf 2-connected} if its incidence   graph   is  2-connected.
\end{definition}

Recall that Dirac's Theorem states that every $n$-vertex graph with minimum degree at least $n/2$ is hamiltonian. This degree is much less than
one needs to guarantee that an $n$-vertex hypergraph has a 
 Berge hamiltonian cycle.

\begin{const}\label{con3}
Let $V(\cH)=V_1 \cup V_2$ where $|V_1| =\lfloor (n+1)/2 \rfloor$, $|V_2| = \lceil (n+1)/2 \rceil$, $V_1 \cap V_2 = \{v\}$, and  let $E(\cH)$ consist of 
 all sets of size at least 2 contained either  in $V_1$ or in $V_2$. Then $\cH$ has minimum degree $2^{\lfloor (n+1)/2 \rfloor-1} - 1$ but no Berge hamiltonian cycle, since the incidence graph has a cut vertex.
If $\cH'$ is formed by the edges of $\cH$ of size $n/4$, then this $n/4$-uniform hypergraph still has an exponential in $n$ minimum degree.
\end{const}

\begin{const}\label{con4}
Let $V(\cH)=V_1 \cup V_2$ where $|V_1| =\lceil (n+2)/2 \rceil$, $|V_2| = \lfloor (n-2)/2 \rfloor$, $V_1 \cap V_2 = \emptyset$, and  let $E(\cH)=E_1\cup E_2$, where $E_1$ is the set of all subsets $A$ of $V(\cH)$ of size $\lceil n/4\rceil$ such that
$|V_1\cap A|=1$ (and $|V_2 \cap A| = \lceil n/4 \rceil - 1$), and $E_2=\{V_1\}$. 
Then $\cH$ has an exponential in $n$  minimum degree, high connectivity and positive codegree of each pair of the vertices. But again, 
$\cH$ has no Berge hamiltonian cycle.
\end{const}

Thus without additional restrictions, the bounds on minimum degree, in terms of the number of vertices of a hypergraph, that guarantee a Berge hamiltonian cycle are far from linear.
On the other hand, by translating Theorems~\ref{mainj} and~\ref{jackson2} into the language of hypergraphs, we obtain the following lower bounds on the minimum degree in terms of {\em the number of vertices and edges} guaranteeing existence of hamiltonian cycles.

\begin{thm}\label{mainj2} Fix integers $n,m,\delta$ such that 
\begin{equation}\label{nm}
\mbox{\em $ \delta \geq n$ and $\delta \geq \frac{m+5}{3}$.}
\end{equation}
 If $\cH$ is a 2-connected
$n$-vertex hypergraph with $m$ edges and minimum degree at least $\delta$, then $\cH$  has a hamiltonian Berge cycle. \end{thm}

\begin{thm}\label{jackson22} Let $\delta \geq n$ and $\delta\geq (m+1)/2$. If an 
$n$-vertex hypergraph $\cH$ with $m$ edges and minimum degree at least $\delta$  has no hamiltonian Berge cycle,
 then 
the incidence graph $I(\cH)$ is
 either $G_1(n)$ 
in Example~\ref{con1} or $ G_2(a,b)$ for some  $a$ and $b$ with $a + b = n$ in Example~\ref{con2}. 
\end{thm}

Examples~\ref{con1} and~\ref{con2} show the sharpness of Theorem~\ref{jackson22}, and 
Example~\ref{con5}  shows the sharpness of Theorem~\ref{mainj2}.

On the other hand, Condition~(\ref{nm}) in  a 2-connected
$n$-vertex hypergraph $\cH$ with $m$ edges and minimum degree at least $\delta$ implies more than simply a hamiltonian cycle
in~$\cH$. We discuss this in the next section.

\subsection{Pancyclic graphs and hypergraphs}
An $n$-vertex graph is  {\em pancyclic} if it contains a cycle of length $\ell$ for every $3 \leq \ell \leq n$. 
There are many results on pancyclic graphs. For instance, Bondy~\cite{Bondy} proved that every Hamiltonian graph with at least $n^2/4$ edges is either the balanced complete bipartite graph or is pancyclic. Later, Bondy~\cite{Bondy2} made the metaconjecture that every nontrivial condition that  implies that a graph $G$ is hamiltonian also implies that $G$ is pancyclic. See~\cite{pansurvey} for more results on pancyclic graphs.
In this paper, we consider similar  notions for hypergraphs. 

\begin{definition}
A hypergraph $\cH$ is {\bf pancyclic} if it contains a Berge cycle of length $\ell$ for every $\ell \geq 3$. Furthermore, we say that $\cH$ is {\bf super-pancyclic} if for every   $A\subseteq V(\cH)$ with $ |A|\geq 3$, $\cH$ has  a Berge cycle whose set of base vertices is  $A$.
\end{definition}
We can similarly define super-pancyclic bipartite graphs.

\begin{definition} A bipartite graph $G$ with partition $(X,Y)$ is {\bf $X$-super-pancyclic} if for every $X' \subseteq X$ with $|X'| \geq 3$, 
$G$ has a cycle $C$  with $V(C) \cap X = X'$. 
\end{definition}

An interesting result on pancyclic hypergraphs was proved by
Lu and Wang~\cite{LW}:
\begin{thm}[Lu, Wang~\cite{LW}]
Let $R$ be a finite set of integers. Then there exists some integer $n_0=n_0(R)$ such that for every $n \geq n_0$, each 
  $n$-vertex hypergraph $\cH$ with edge cardinalities in $R$ and minimum co-degree at least 1  is pancyclic. Furthermore, for each set of vertices $A$ with $|A| \geq n_0$, $\cH$ contains a Berge cycle with the set of base vertices $A$.
\end{thm}



Also, Jackson's Theorem (Theorem~\ref{jackson}) implies the following result on pancyclic hypergraphs.

\begin{thm} [Hypergraph version of Theorem~\ref{jackson}] Suppose $\delta\geq n$ and $ \delta\geq (m+2)/2$. Then every   $n$-vertex hypergraph with $m$ edges and minimum degree  at least $\delta$ is super-pancyclic.
\end{thm}

One of the main goals of this paper is to extend this result to more general hypergraphs.


If  $A$ is a subset of the part $X$ in a graph $G \in \cG(n,m,\delta)$ and $C$ is a cycle in $G$ with $V(C)\cap X=A$ and $V(C)\cap Y=B$, then 
$G[A\cup B]$ is $2$-connected. Thus
 every $X$-super-pancyclic bipartite graph satisfies:
 \begin{equation}\label{lll}
 \parbox{14cm}{\em
 For each $A\subseteq X$ with $|A|\geq 3$, there is $B\subseteq Y$ with $|B|\geq |A|$ such that $G[A\cup B]$ is $2$-connected. }
 \end{equation}

We will show  that this necessary  condition for $G$ to be $X$-super-pancyclic is also sufficient when $m\leq 3\delta-5$.

\begin{thm}\label{mainpan} Let $ \delta \geq n$ and $m\leq 3\delta-5$. If $G \in \cG(n,m,\delta)$ satisfies~(\ref{lll}),
then $G$ is $X$-super-pancyclic.
\end{thm}

In terms of hypergraphs our result is as follows

\begin{cor}[Hypergraph version of Theorem~\ref{mainpan}]\label{mainpan2} Let $ \delta \geq n$ and $m\leq 3\delta-5$. If the incidence graph of
 an $n$-vertex hypergraph $\cH$  with $m$ edges and minimum degree $\delta(\cH)$ satisfies~(\ref{lll}),
 then $\cH$ is super-pancyclic.
\end{cor}

It could be that the necessary condition~(\ref{lll}) for  $\cH$ to be super-pancyclic is  sufficient not only for $m\leq 3\delta-5$.
It would be interesting to find the
 range of $m$ for which~(\ref{lll}) yields the conclusion of
Corollary~\ref{mainpan2}.

\medskip

We present the main proofs in the language of bipartite graphs. In Section 2, we introduce
the notion of a tight pair and prove some properties of such pairs.
 Using these properties  we prove Theorems~\ref{mainj} and~\ref{jackson2} in Section~3. In Section~4, we prove Theorem~\ref{mainpan} and the reduction to Corollary~\ref{mainpan2}.

\section{Tight pairs and their properties} 


\begin{definition} Let $G \in \cG(n,m,\delta)$.  A {\bf tight pair} in $G$ is a pair $(C,x)$ where $C$ is a longest cycle in $G$ and $x\in X-C$ 
is such that $|N(x)\cap v(C)|$  is maximum over all pairs $(C',x')$  where $C'$ is a longest cycle in $G$ and $x'\in X-V(C')$.
\end{definition}

By definition, a graph $G \in \cG(n,m,\delta)$ has a tight pair if and only if  $G$ does not contain a cycle of length $2n$. In this section, $G$
is always a graph in $\cG(n,m,\delta)$ and $(C,x)$ is a tight pair where 
 $C = y_1x_1 \ldots y_\ell x_\ell y_1$. 

\begin{claim}\label{neighbor} 
  If $y_i \in N_C(x)$, then $N(x_i) - V(C)$ and $N(x)-V(C)$ are disjoint.
\end{claim}

\begin{proof}
Suppose $y' \in (N(x_i) \cap N(x)) - V(C)$.  Then the cycle
 \[y_1 x_1 \ldots y_i x y' x_i y_{i+1} x_{i+1} \ldots y_\ell x_\ell y_1\] contains $2(\ell + 1)$ vertices, contradicting the choice of $C$.
\end{proof}

Similarly, by reorienting the cycle, we also obtain $N(x_{i-1}) - V(C)$ and $N(x)-V(C)$ are disjoint.

\begin{claim}\label{neighbor3} 
  If $y_i, y_j \in N_C(x)$, then $N(x_{i})-V(C)$ and $N(x_j) - V(C)$ are disjoint.
\end{claim}

\begin{proof}
Suppose $i < j $ and $y' \in (N(x_i) \cap N(x_j)) - V(C)$.  Then the cycle 
 \[y_1 x_1 \ldots y_i x y_j x_{j-1} y_{j-1} \ldots x_i y' x_j y_{k+1} \ldots y_\ell x_\ell y_1\] contains more vertices than $C$, a contradiction.
\end{proof}

 For $1 \leq i,j \leq \ell$,  let $C[i,j]$ denote the clockwise path segment of $C$ from $y_i$ to $y_j$ (where we take indices of the path modulo $\ell$ if $i > j$).

\begin{definition}
For $1 \leq i < j \leq \ell$ we say that vertices $x_i$ and $x_j$ are {\bf crossing} in $C$ if there exists indices ${i'}, {j'} \in [\ell]$ such that $y_{i'} \in N(x_i)$, $y_{j'} \in N(x_j$), and either
\begin{itemize}
\item $i' = j'+1$ with $i+1 \leq j' \leq j-1$ (so $y_{i'}, y_{j'} \in V(C[i+1, j])$, or
\item $j'=i'+1 \mod \ell$ and $y_{i'}, y_{j'} \in V(C[j+1, i])$.
\end{itemize}
\end{definition}
Note that we do not consider $\{i', j'\} = \{i,{i+1}\}$ or $\{i', j'\} = \{j, {j+1}\}$. 

\begin{lem}\label{crossing} Let $1\leq i<j\leq \ell$. 
 If $x_i$ and $x_j$ are not crossing, then \[|N(x_i) \cap V(C)| + |N(x_j) \cap V(C)| \leq |V(C) \cap Y| +2.\]
\end{lem}

\begin{proof}Let $P_1 = y_{i+1}x_{i+1} \ldots y_j$. 
Define \[I_1:= \{k \in \{i+1, \ldots, j-1\}: y_k\in N(x_j), y_{k+1} \in N(x_i)\}.\] By pigeonhole principle, if $|N(x_i) \cap V(P_1)| + |N(x_j) \cap V(P_1)| \geq |V(P_1) \cap Y| + 2$ then $I_1$ is nonempty, which implies that $x_i, x_j$ are crossing in $C$, a contradiction.

Similarly let $P_2=y_{j+1}x_{j+1} \ldots y_\ell x_\ell y_1 \ldots x_{i-1}y_i$ and \[I_2:=\{k \in \{j+1, \ldots, \ell\} \cup \{1, \ldots, {i-1}\}: y_k \in N(x_i), y_{k+1} \in N(x_j)\}.\]

We obtain $|N(x_i) \cap V(P_2)| + |N(x_j) \cap V(P_2)| \leq |V(P_2)  \cap Y| + 1$, otherwise $x_i, x_j$ are crossing.

Hence 
\begin{eqnarray*}
&& |N(x_i) \cap V(C)| + |N(x_j) \cap V(C)|\\
&\leq & |N(x_i) \cap V(P_1)| + |N(x_j) \cap V(P_1)| 
+  |N(x_i) \cap V(P_2)| + |N(x_j) \cap V(P_2)|\\
& \leq & |V(P_1)  \cap Y|+1  + |V(P_2) \cap Y| +1 
\leq  |V(C) \cap Y| + 2.
\end{eqnarray*}
\end{proof}

\begin{lem}\label{lemT} 
 If $2 \leq |N(x ) \cap V(C)| < |Y\cap V(C)|$ and $n=|X|\leq \delta$, then $m = |Y| \geq 3\delta-4$. 
\end{lem}
\begin{proof} Let $T$ be the set of vertices in $N(x )\cap V(C)$ and set $|T|=t$. Note all vertices of $T$ are in $Y$. By Claims~\ref{neighbor} and~\ref{neighbor3}, the neighborhoods all vertices in $\{x \} \cup \{x_i: y_i \in T\}$ are disjoint  outside of $C$. Let $\ell-s$ be the maximum size of $|N(x_i)\cap V(C)|$ for $x_i$ such that $y_i\in T$.

We will first show that $s>0$. If not, then without loss of generality, suppose $y_1\in T$ and ${V(C)\cap Y\subseteq N(x_1)}$. For each $x_j$ such that $y_j\in T$ and $j > 1$, we can form the cycle $$C':=y_1x y_j x_{j-1} y_{j-1} \ldots y_{2}x_1y_{j+1} x_{j+2} \ldots y_\ell x_\ell y_1.$$

The cycle $C'$ has the same length as $C$, and contains all of $C$ except $x_j$. By the choice of $(C,x)$, ${|N(x_j)\cap V(C')|\leq t}$. Thus, $|N(x_j)\cap V(C)|\leq t$ for every $x_j$ such that $y_j\in T$ and $j > 1$.

Since $t\leq |Y\cap V(C)|-1=\ell-1$ and $\ell\leq |X|-1\leq  \delta-1$, we have $\delta\geq t+2$.

So 
\begin{align*}
|Y|\geq |N(x) \cup \Big(\bigcup_{y_i\in T} N(x_i)\Big)|  &\geq (t+1)\delta-d_C(x_1)-\Big[\sum_{y_i\in T, i\neq 1}d_C(x_i)\Big]-d_c(x)+|V(C)\cap Y|\\
&=(t+1)\delta-\ell-(t-1)t-t+\ell\\
    &=3\delta+(t-2)\delta-t^2\\
    &\geq 3\delta+(t-2)(t+2)-t^2=3\delta-4.
\end{align*}
 Hence we may assume $s \geq 1$.

\textbf{Case 1}:
There are no crossing vertices $x_i, x_j$, with $y_i,y_j\in T$. So by Lemma~\ref{crossing}, for each  $x_i$ and $ x_j$, with $y_i,y_j\in T$, $d_C(x_i)+d_C(x_j)\leq \ell+2$.

If $t$ is even, we arbitrarily pair up vertices in $T$ and apply Lemma~\ref{crossing} to obtain that the vertices in $T$ together have degree sum in $C$ at most $t (\ell+ 2)/2$. Therefore  $$|Y|\geq (t+1)\delta-\frac{t}{2}(\ell+2)-t+\ell.$$  Recall $t\leq \ell-1$,  $\ell\leq \delta-1$, and $\delta\geq t+2$ (implying $\delta\geq 4$). Therefore \begin{align*}
 |Y|\geq   (t+1)\delta-\frac{t}{2}(\ell+2)-t+\ell&= (t+1)\delta-\ell(\frac{t}{2}-1)-2t\\
    &\geq  (t+1)\delta-(\delta-1)(\frac{t}{2}-1)-2t\\
    &\geq (\frac{t}{2}+2)\delta-\frac{3}{2}t-1\\
    &=3\delta+(\frac{t}{2}-1)\delta-\frac{3}{2}t-1\\
    &\geq 3\delta+(\frac{t}{2}-1)4-\frac{3}{2}t-1\\
    &\geq 3\delta+\frac{1}{2}t-5\\
    &\geq 3\delta-4.
\end{align*}

If $t$ is odd, fix $y_i\in T$ such that $d_C(x_i)$ is minimum. By Lemma~\ref{crossing}, $d_C(x_i) \leq (\ell + 2)/2$. The other $t-1$ vertices in $T$ can be arbitrarily paired up as in the previous case so that their total degree sum in $C$ is at most $(t-1) (\ell + 2)/2$. Therefore 
$$|Y|\geq (t+1)\delta-\frac{t-1}{2}(\ell+2)-\frac{\ell + 2}{2}-t+\ell,$$  as in the case of $t$  even. Thus $|Y|\geq 3\delta-4$.


\textbf{Case 2}: There are crossing vertices $x_i$ and $x_j$, such that $y_i,y_j\in T$ for some $1\leq i<j\leq \ell$. Suppose that for some $i+1 \leq k <j$, $y_{k+1} \in N(x_i)$ and $y_{k} \in N(x_{j})$ (the other case is similar so we omit it). 

We will show $N(x_k)\cap N(x )\subset C$. Suppose not. Let $y'\not \in C$ such that $y' \in N(x_k)\cap N(x )$. Then the cycle 
$$y_{i+1}\dots x_ky'x y_i y_{i-1}\dots y_{k+1}x_iy_{i+1}$$
 is longer than $C$, a contradiction.

Next, we will show $N(x_k)\cap N(x_q)\subset C$, for $x_q$ such that $y_q\in T$ and $q\neq k$. Suppose not. Then there is some vertex $y'\not \in C$ such that $y' \in N(x_k)\cap N(x_q)$. If $x_q\in C[j,i]$, then we have the cycle $$y_{i+1}\dots x_ky'x_qy_{q+1}\dots y_ix y_qx_{q-1}\dots y_{k+1}x_iy_{i+1}.$$
If $x_q\in C[i,k]$, then we have the longer cycle $$x_q y_{q+1}\dots y_k x_j y_{j+1}\dots y_q x  y_jy_{j-1}\dots x_k y'x_q.$$
If $x_q\in C[k+1,j]$, then we have the longer cycle $$y_{i+1}\dots x_ky'x_qy_{q+1}\dots y_i x  y_q x_{q-1}\dots y_{k+1}x_iy_{i+1}.$$

In any case we get a cycle with at least $2(\ell+1)$ vertices, a contradiction.

Next, we will show that $|N(x_k)\cap V(C)|\leq t$. The cycle $$C':=y_{i+1}\dots y_k x_j y_{j+1}\dots y_ix y_j x_{j-1}\dots y_{k+1} x_iy_{i+1}.$$
has $2\ell$ vertices, includes $x $, and includes all the vertices of $C$ except $x_k$.   Since
$(C,x)$ is a tight pair, $|N(x_k)\cap V(C')|\leq t$. So $|N(x_k)\cap V(C)|\leq t$. Now, consider the vertices $x ,x_k$ and all $x_q$ such that $y_q\in T.$

If $y_k\in T$, then we know $x_k$ has at most $t$ neighbors on $C$. Additionally, every other vertex $x_q$ such that $y_q\in N(x )$ has at most $\ell-s$ neighbors in $C$. So  $$|Y|\geq (t+1)\delta-(t-1)(\ell-s)-t-t+\ell.$$ 

Since $s\geq 1$, 
\begin{align*}
|Y|\geq    (t+1)\delta-(t-1)(\ell-s)-t-t+\ell&=(t+1)\delta-(t-2)\ell+s(t-1)-2t\\
   &\geq (t+1)\delta-(t-2)(\delta-1)+s(t-1)-2t\\
   &=3\delta+(t-2)+s(t-1)-2t\\
   &\geq 3\delta+t-2+t-1-2t\\
   &= 3\delta-3.
\end{align*}

If $y_k\not \in T$, then by the choice of $x $, $x_k$ has at most $t$ neighbors on $C$ and shares no neighbors outside $C$ with $x $ and $x_q$ such that $y_q\in T$.
So $|Y|\geq (t+2)\delta-t(\ell-s)-t-t+\ell$. This is greater than $(t+1)\delta-(t-1)(\ell-s)-t-t+\ell$ which is at least $3\delta-3$.
\end{proof}

\begin{lem}\label{lemT2} 
If $n \leq \delta$ and $Y\cap V(C) \subseteq N(x )$, then for each $x_i \in X \cap V(C)$ and each $y \in N(w) - V(C)$, $w$ separates $y$ from $V(C) - w$. 
\end{lem}

\begin{proof}
Without loss of generality, suppose $x_1 \in C$ does not separate $y\in N(x_1)$ from $V(C) - x_1$. Then $G -x_1$ contains a path $P$ from $y$ to $V(C) - x_1$. Let $z$ be the endpoint of $P$ in $V(C)$, say $z \in \{x_i,y_i\}$ for some $1 < i \leq \ell$. First suppose $x  \notin V(P)$. If $z\neq y_\ell$, then either $y_1 x_1 y P z y_{i} x_{i-1} \ldots y_2 x  y_{i+1} x_{i+1} \ldots y_\ell x_\ell$ or $y_1 x_1 y P z x_{i-1} y_{i-1} \ldots y_2 x  y_{i+1} x_{i+1} \ldots y_\ell x_\ell$ is a longer cycle than $C$. If $z=y_\ell$, then $y_2x_1Py_\ell x_\ell y_1 x y_{\ell - 1} x_{\ell - 2}\ldots y_2$ is a longer cycle. Otherwise, let $P'$ be the segment of $P$ from $y$ to $x $. Then we instead take the cycle $y_1x_1yP'x  y_2 x_2 \ldots y_\ell x_\ell y_1$. 
\end{proof}

\section{Proofs of Theorems~\ref{mainj} and~\ref{jackson2}}

%

{\em Proof of Theorem~\ref{mainj}}. Suppose a longest cycle of $G$ has length $2\ell$ where $\ell < n$ and fix a tight pair $(C, x )$. Say $C = y_1 x_1 \ldots y_\ell x_\ell y_1$, and set $t := |N(x ) \cap V(C)|$.  

{\bf Case 1}: $t \leq 1$. Because $G$ is 2-connected, there exist paths $P_1$ and $P_2$ from $x $ to $V(C)$ such that $P_1$ and $P_2$ are disjoint except at $x $, and both $P_i$ are internally disjoint from $V(C)$.  Let $z_1$ and $z_2$ be the endpoints of $P_1$ and $P_2$ respectively that are not $x $. Among all such paths, choose $P_1$ and $P_2$ so that $|\{z_1, z_2\} \cap Y|$ is maximum. Often we will use $P_1 \cup P_2$ to refer to the combined path from $z_1$ to $z_2$.  

{\bf Case 1.1}: Either $z_1 \in Y$ or $z_2 \in Y$. Without loss of generality, let $z_1 = y_1$. Note that this is exactly the case where $t = 1$, because the second vertex in $P_1$ is a vertex of $X$ not in $V(C)$, so we may choose this vertex to be $x $ and $P_1$ to be the path $x y_1$. Since $t < 2$, $|(P_2 \cap Y) - V(C)| \geq 1$.

Say $z_2 \in \{x_j, y_j\}$ for some $1 < j\leq \ell$. Also, $j \neq \ell$, otherwise we could replace the vertex $x_\ell$ in $C$ with the path $P_1 \cup P_2$ to obtain a longer cycle. 

By Claim~\ref{neighbor}, we have that $N(x_1)$ and $N(x )$ do not intersect outside of $C$. Furthermore, we also claim that $N(x_1)$ does not intersect $N(x_{j+1})$ outside of $C$. Suppose that ${y' \in (N(x_1) \cap N(x_{j+1})) - V(C)}$. 

If $y' \in V(P_1 \cup P_2)$, then let $P'$ be the segment of $P_1 \cup P_2$ from $y_1$ to $y'$. Then 
\[ y_1 P' y' x_1 y_2 \ldots y_\ell x_\ell y_1\] is a cycle with at least $2(\ell + 1)$ vertices, a contradiction.

So we may assume that $y' \notin V(C) \cup V(P_1 \cup P_2)$. Then \[C':= y_1 P_1 \cup P_2 z_2 y_j x_{j-1} y_{j-1} \ldots y_2 x_1 y' x_{j+1} y_{j+2} \ldots y_\ell x_\ell y_1\] is a cycle which contains $(V(C) \cap Y) - y_{j+1}$ and $y' \cup (V(P_1 \cup P_2) \cap Y)$. That is, $|V(C')| \geq 2(\ell + 1)$, a contradiction.

Next, we show that $x_1$ and $x_{j+1}$ are not crossing. Suppose first that for some $j+1 < k \leq \ell$, $y_k \in N(x_1)$ and $y_{k+1} \in N(x_{j+1})$. Then 
\[y_1 P_1 \cup P_2  z_2 y_j x_{j-1} y_{j-1} \ldots y_2x_1 y_k x_{k-1} \ldots y_{j+2}x_{j+1} y_{k+1} x_{k+1} \ldots y_\ell x_\ell y_1\]
is at least as long as $X$, contains $y_k$ and $y_{k+1}$, and does not contain $x_{k}$. Similarly, if $x_{k+1} \in N(x_1)$ and $x_{k} \in N(x_{j+1})$ for some $2 \leq k \leq j$, then  
\[y_1 P_1 \cup P_2  z_2 y_j x_{j-1} y_{j-1} \ldots y_{k+1}x_1 y_2 \ldots x_{k-1}y_k x_{j+1} y_{j+2} \ldots y_m x  y_1\] is a cycle with the same properties. Hence by the choice of $(C, x )$ we have that $N(x_{k})$ contains at most $t = 1$ vertices in $C$. But $N(x_{k}) \supseteq \{y_k, y_{k+1}\}$, a contradiction.

Therefore by Lemma~\ref{crossing}, $d_C(x_1) + d_C(x_{j+1}) \leq \ell + 2$. 

Note that 
\begin{equation}\label{j+1} \mbox{$N(x )$ and $N(x_{j+1})$ have at most one common vertex outside of $C$.}
\end{equation} To see this, suppose that $\{y', y''\} \subseteq (N(x ) \cap N(x_{j+1})) - V(C)$. If both $y', y'' \in V(P_1 \cup P_2)$, with say $y'$ appearing before $y''$ in $P_1 \cup P_2$, then let $P'$ be the segment of $P_1 \cup P_2$ from $y'$ to $z_2$, oriented backwards (from $z_2$ to $y'$). Then \[y_1 x_1 \ldots y_j  z_2 P' y' x_{j+1} y_{j+2} \ldots y_\ell x_\ell y_1 \] is a longer cycle than $C$. Otherwise, we may assume that $y'' \notin V(P_1 \cup P_2)$. Then \[y_1 x_1 \ldots y_{j}  z_2 P_2 x  y'' x_{j+1} y_{j+2} \ldots y_\ell x_\ell y_1\] is a longer cycle.

Putting it all together, we get
\begin{eqnarray*}|Y|\geq |N(x_1) \cup N(x_{j+1})\cup N(x )| & \geq & 3\delta - (d_C(x_1) + d_C(x_{j+1}) + d_C(x ))\\
& & - |(N(x_{j+1}) \cap N(x )) - V(C)| + |V(C)\cap Y|\\
&\geq & 3\delta - (\ell + 2 + 1) - 1 + \ell\\
& = & 3\delta-4,
\end{eqnarray*} a contradiction.

{\bf Case 1.2}: Both $z_1 \in X$ and $z_2 \in X$. Note that this implies $t = 0$. For simplicity, let $z_1 = x_1$ and $z_2 = x_j$ for some $2 \leq j \leq \ell$.  We must have that $|(V(P_1 \cup P_2)) \cap Y| \geq 2$ since $G$ is bipartite.

We claim first that $N(x_2)\cap N(x_{j+1})\subseteq V(C)$. Suppose that ${y' \in (N(x_2) \cap N(x_{j+1})) - C}$. If $y' \in V(P_1 \cup P_2)$, then let $P_1'$ be the segment of $P_1 \cup P_2$ from $x_1$ to $y'$, let $P_2'$ be the segment of $P_1 \cup P_2$ from $x_j$ to $y'$, and let $P'$ be the longer of the paths $P_1'$ and $P_2'$. Note that $P'$ must contain $x $ and at least 2 vertices from $Y$. Then either
\[y_1 x_1 P' y' x_2 y_3 \ldots y_\ell x_\ell x_1 \;\; \text{ or } \;\; y_1 x_1 \ldots y_j P' y' x_{j+1} y_{j+2} \ldots y_\ell x_\ell y_1\]
is a cycle with at least $2(\ell - 1 + 2)$ vertices, a contradiction.

Next, suppose $y' \notin V(P_1 \cup P_2)$. Then the cycle \[y_1 x_1 P_1 \cup P_2 x_j y_j x_{j-1} \ldots y_3 x_2 y' x_{j+1}y_{j+2} \ldots y_\ell x_\ell y_1\] contains at least $2(\ell - 2 + 2 + 1)$ vertices. This proves that $N(x_2)$ and $N(x_{j+1})$ are disjoint outside of $C$. 
The same proof as for~\eqref{j+1} shows that also $N(x_2)$ and $N(x )$ intersect at at most one vertex outside of $C$, and same for $N(x_{j+1})$ and $N(x )$. 

Finally, we show that $x_2$ and $x_{j+1}$ are not crossing. Otherwise, if there exists some $j+1 < k \leq \ell$, with $y_k \in N(x_2)$ and $y_{k+1} \in N(x_{j+1})$. Then the cycle
\[C':= y_1 P_1 \cup P_2 y_j x_{j-1} y_{j-1} \ldots y_3x_2 y_k x_{k-1} \ldots y_{j+2}x_{j+1} y_{k+1} x_{k+1} \ldots y_\ell x_\ell y_1\] contains at least as many vertices as $C$. Furthermore $V(C) - \{y_2, y_{j+1}\} \subset V(C')$, and $x_{k} \notin V(C')$. But $|N(x_k) \cap C'| \geq 2$, contradicting the choice of $(C, x )$. The case where $x_{k+1} \in N(x_1)$ and $x_{k} \in N(x_{j+1})$ for some $2 \leq k \leq j$ is similar so we omit the proof. By Lemma~\ref{crossing}, this implies $d_C(x_2) + d_C(x_{j+1}) \leq \ell + 2$. 

Thus 
\begin{eqnarray*}|Y|\geq |N(x_1) \cup N(x_{j+1})\cup N(x )| & \geq & 3\delta - (d_C(x_2) + d_C(x_{j+1}) + d_C(x )) - |(N(x_{2}) \cap N(x )) - V(C)|\\
&& -|(N(x_{j+1}) \cap N(x )) - V(C)| + |V(C)\cap Y|\\
&\geq &3\delta - (\ell + 2 + 0) - 1 -1 + \ell \\
&= &3\delta-4,\end{eqnarray*}
a contradiction.

{\bf Case 2}: $2 \leq t \leq \ell - 1$. Apply Lemma~\ref{lemT} to obtain $n \geq 3\delta-4$, a contradiction.

%

{\bf Case 3}: $t = \ell$. By Lemma~\ref{lemT2}, every vertex in $X \cap V(C)$ is a cut vertex of $G$, a contradiction. \qed

\bigskip 
{\em Proof of Theorem~\ref{jackson2}}. Suppose $G \in \cG(n,2\delta-1, \delta)$ contains no cycle of length $2n$.
If $n=2$ and $G$ has no cycles, then $G=G_2(1,1)$. 
 If $n\geq 3$, since $m\leq 2\delta-1$, $G$ has a cycle. Let $(C,x )$ be a tight pair, and set $t:= |V(C) \cap N(x )|$.

{\bf Case 1}: $t = 0$. Since $|Y| = 2\delta-1$, for every $x, x' \in X$, $|N(x) \cap N(x')| \geq 1$. In particular, for each $x_i \in V(C) \cap X$, $N(x_i) \cap N(x )$ contains at least 1 vertex, say $y'_i \in Y - V(C)$. 

If for some $1 \leq i \leq \ell$, $y'_i \neq y'_{i+1}$ (indices taken mod $\ell$), then the cycle $y_1 x_1 \ldots y_{i} x_i y'_i x  y'_{i+1} x_{i+1} \ldots y_\ell x_\ell y_1$ contains more vertices than $C$. Otherwise, the cycle $C':=y_1 x_1 \ldots y_i x_i y'_{i+1} x_{i+1} \ldots y_\ell x_\ell y_1$ is also a longest cycle of $G$, but $|V(C') \cap N(x )| > 0$, contradicting the choice of $(C, x )$ as a tight pair.

{\bf Case 2}: $t = 1$. Without loss of generality, let $N(x ) = y_1$. We have that $N(x ) - V(C)$ contains at least $\delta-1$ vertices. If $N(x ) \cap N(x_1)$ contains a vertex $y \in Y - V(C)$, then we get the longer cycle $y_1 x  y x_1 y_2 \ldots y_\ell x_\ell x_1$. So $N(x_1) \cap N(x ) = \{y_1\}$, and furthermore since ${|Y| = 2\delta-1}$, ${Y = N(x_1) \cup N(x )}$, and $Y\cap V(C) \subseteq N(x_1)$. Note that by symmetry, $N(x_\ell) = N(x_1)$.  Suppose that for some $1 < i < \ell$, there exists $y \in (N(x ) \cap N(x_i)) - y_1$. Then we obtain the longer cycle $y_1 x_\ell y_\ell \ldots x_{i+1}y_{i+1} x_1 y_2 \ldots y_{i}x_i y x  y_1$. This shows that $N(x_i) = N(x_1)$ for all $1 \leq i \leq \ell$. 

By the choice of $(C, x )$ as a tight pair, $|N(x') \cap V(C)| \leq 1$ for all $x' \in X - V(C)$. Since ${Y = N(x_1) \cup N(x )}$, ${N(x) - V(C) = N(x ) - y_1}$. Hence the vertices of $V(C)$ induce a complete bipartite graph,  as do the vertices of $V(G) - V(C)$. Furthermore, every $x \in X - V(C)$ contains exactly one neighbor in $V(C)$. It is easy to show that $G$ contains a cycle of length $2n$ unless for every $x \in X - V(C)$, $N(x) \cap V(C) = \{y_1\}$. Therefore $G$ is isomorphic to $G_2(n - \ell, \ell)$.

{\bf Case 3}: $2 \leq t \leq \ell - 1$. If  $\delta\geq 4$, then by Lemma~\ref{lemT} we get $|Y| \geq 3\delta-4$, a contradiction. Suppose $\delta=3$ and $G$ has no cycle of length $2n$. Since $G$ contains a cycle, $n\neq 2$ and so $n=3$. It is easy to check that $G=G_1(n)$ for the case $t \geq 2$.

{\bf Case 4}: $t = \ell$. By Lemma~\ref{lemT2},  for any $x_i \in X \cap V(C)$ and each $y \in N(x_i) - V(C)$, $x_i$ separates $y$ from $V(C) - x_i$. In particular, this implies that $y \notin N(x_j)$ for any $j \neq i$. Since $|V(C) \cap Y| = \ell \leq n-1 \leq \delta-1$, each $x_i$ has at least $\delta-\ell$ neighbors outside of $V(C)$ that are shared by no other $x_j$. Then $|Y| \geq \ell + (\ell + 1)(\delta-\ell) = \ell \delta - \ell^2 + \delta = \delta +\ell (\delta- \ell)  \geq \delta + (\delta-1)(1)$, where equality holds only if $\ell = \delta-1$ (so $\delta=n$), and each vertex in $X$ has exactly one neighbor outside of $C$. That is, $G$ is isomorphic to $G_1(n)$. 
\qed

\section{Proofs of Theorem~\ref{mainpan} and Corollary~\ref{mainpan2} }

%
{\em Proof of Theorem~\ref{mainpan}}. Fix $G \in \cG(n,m,\delta)$ with partition $(X,Y)$. We will show that for every $X' \subseteq X$ with $|X'| \geq 3$, there exists a cycle $C$ in $G$ such that $V(C) \cap X = X'$. 

We proceed by induction on $|X'|$. For the base case, suppose $|X'| = 3$. Let $Y' \subseteq Y$ such that $G':=G[X' \cup Y']$ is 2-connected. It is easy to check, via an ear-decomposition argument for instance, that a bipartite 2-connected graph has a cycle of length at least 6 unless one of its parts has size 2. But we have  $|Y'| \geq |X'| \geq 3$.
Hence  for each $X'\subset X$ with $|X'| = 3$, $G$ contains a cycle $C$ of length 6 with $V(C)\cap X=X'$. This yields the following facts.

\begin{claim}\label{codegree}(i) For each $x, x' \in X$, $|N(x) \cap N(x')| \geq 1$. 

\hspace{21mm}(ii) For any $x, x', x'' \in X$, ${|N(x) \cap (N(x') \cup N(x''))| \geq 2}
$.\end{claim}
\begin{proof}
For (i), let $x'' \in X - \{x,x'\}$. Then $G$ contains a cycle, say $y_1 x y_2 x' y_3 x'' y_1$, so ${y_2 \in N(x) \cap N(x')}$. For (ii), using the same cycle, we have $N(x) \cap (N(x') \cup N(x'')) \supseteq \{y_1, y_2\}$. 
\end{proof}

For the induction step, fix $X' \subseteq X$ with $k:=|X'| \geq 4$ and consider the subgraph ${G_{X'} = G[X' \cup N(X')]}$. We wish to show that $G_{X'}$ contains a cycle using all of $X'$. Suppose not. Then by induction hypothesis, for each $x \in X'$, $G$ contains a cycle $C_x$ such that $V(C_x) \cap X = X' - x$. In particular, this cycle must be contained in $G_{X'}$.

Among all $x \in X'$, pick $x$ such that $|N(x) \cap V(C_x)|$ is maximum. Let $C_x = y_1 x_1 \ldots y_{k-1}x_{k-1}y_1$ and $t := |N(x) \cap V(C_x)|$

{\bf Case 1}: $t=0$. By Part (i) of Claim~\ref{codegree},  $|N(x) \cap N(x_i)| \geq 1$ for all $1 \leq i \leq k-1$. If there exists some $y'_i$, $y'_{i+1}$ such that $y'_j \in N(x) \cap N(x_j)$ for $j \in \{i, i+1\}$ and $y'_i \neq y'_{i+1}$, then $G_{X'}$ contains the cycle

\[C':= y_1 x_1 \ldots x_i y'_{i} x y'_{i+1} x_{i+1}y_{i+2} \ldots y_{k-1} x_{k-1} y_1\] with $2(k-1-1 + 2)$ vertices. I.e., $V(C') \cap X = X'$. 

Otherwise, suppose for every $i$, $N(x_i) \cap N(x) = \{y'\}$. Then $N(x) \cap (N(x_1) \cup N(x_2)) = \{y'\}$, contradicting Part (ii) of Claim~\ref{codegree}.

{\bf Case 2}: $t=1$. Without loss of generality, say $y_{k-1} \in N(x)$. By Part (ii) of Claim~\ref{codegree}, there exists some $y' \in N(x) \cap (N(x_{k-2}) \cup N(x_{k-1}))$ such that $y' \neq y_{k-1}$. Since $t=1$, $y' \notin V(C_x)$. Then either
\[y_1 \ldots y_{k-2}x_{k-2} y' x y_{k-1} x_{k-1} y_1 \; \text{ or } \; y_1 \ldots y_{k-1} x y' x_{k-1} y_1\] is a cycle in $G_{X'}$ which contains all of $X'$. 

{\bf Case 3}: $2 \leq t \leq k-2$. Apply Lemma~\ref{lemT} to $G'_{X'}$, $C_x$, and $x$ to obtain that $n \geq 3\delta-4$, a contradiction.

{\bf Case 4}: $t = k-1$. We claim that for each  $y \in V(G_{X'}) - V(C)$, $|N_{X'}(y)| = 1$. If not, then there exists some $y \notin V(C)$ and some $1 \leq i < j \leq k$ such that $y \in N(x_i) \cap N(x_j)$. But this contradicts Theorem~\ref{lemT2}. Therefore every 2-connected subgraph of $G$ containing $X'$ is a subgraph of $G[X' \cup V(C)]$. But $|V(C) \cap Y| < |X'|$, a contradiction. \qed

%
%
%
%

\medskip

{\em Proof of Corollary~\ref{mainpan2}}.
Fix an $n$-vertex hypergraph $\cH$ with $\delta(\cH) \geq n$ and at most $3n-5$ edges. Let $I(\cH)$ be the incidence graph of $\cH$.

Then $I(\cH) \in \cG(n, |E(\cH)|, n)$ with $|E(\cH)| \leq 3n-5$. Translating from the language of hypergraphs to bipartite graphs, for each $X' \subseteq X$ with $|X'| \geq 3$, there exists $Y ' \subseteq Y$ with $|Y'| \geq |X'|$ such that $G(\cH)[X' \cup Y']$ is 2-connected. By Theorem~\ref{mainpan}, $G(\cH)$ is $X$-super-pancyclic; equivalently, $\cH$ is super-pancyclic.
\qed

\bigskip
{\Large {\bf Concluding remarks.}}

1. In ~\cite{KL}, a  significantly weaker version of Conjecture~\ref{jacksonconj}(ii) is proved.

2.  Let the {\bf dual incidence graph} of $\cH$ be the bipartite graph $I(\cH)$ with parts $(X, Y)$ where $X = E(\cH)$, $Y= V(\cH)$ such that for $e \in X, v \in Y,$ $ev\in E(I(\cH))$ if and only if the vertex $v$ is contained in the edge $e$ in $\cH$. 

By applying Theorem \ref{mainpan}, we obtain the following.

\begin{cor}
Let $\cH$ be a hypergraph with edge cardinalities at least $r$ such that $e(\cH) \leq r$ and $|V(\cH)| \leq 3r-5$. If the dual incidence graph of $\cH$ satisfies (\ref{lll}), then
 for any set of edges $B$ with $|B| \geq 3$, $\cH$ contains a Berge cycle using exactly the edges of $B$.
\end{cor}

\medskip
{\bf Acknowledgement.}
 We thank Misha Lavrov for pointing our attention to Jackson's Conjecture.


\end{document}